\documentclass{birkmult}

\usepackage{url}
\usepackage{paralist}   % fancy itemize-environments
\DeclareSymbolFont{msbm}{U}{msb}{m}{n}
\DeclareMathSymbol{\emptyset}{\mathord}{msbm}{'077}

\newcommand{\K}{{\sf K}}
\renewcommand{\S}{{\sf S}}
\newcommand{\I}{\mathcal{I}}
\renewcommand{\L}{\mathcal{L}}
\newcommand{\zerovec}{0}

\newcommand{\del}{\partial}
\newcommand{\simp}{\Delta}
\newcommand{\R}{\mathbb R}
\newcommand{\Z}{\mathbb Z}
\newcommand{\N}{\mathbb N}
\newcommand{\dprod}[1]{#1^2_\Delta}
\DeclareMathOperator{\conv}{conv}
\DeclareMathOperator{\id}{id}
\DeclareMathOperator{\sgn}{sgn}

\theoremstyle{plain}
\newtheorem{thm}{Theorem}[section]
\newtheorem{lemma}[thm]{Lemma}
\newtheorem{lemdef}[thm]{Lemma and Definition}
\newtheorem{cor}[thm]{Corollary}
\newtheorem{prop}[thm]{Proposition}
\newtheorem{sub}[thm]{Subsystem}

\theoremstyle{definition}
\newtheorem{definition}[thm]{Definition}

\theoremstyle{remark}
\newtheorem{remark}[thm]{Remark}
\newtheorem{remark*}{Remark}
\newtheorem{notation*}{Notation}

\title{Necessary Conditions for Geometric \\ Realizability of Simplicial Complexes
\index{geometric realizability}}
\author{Dagmar Timmreck}
\address{
Inst. Math., MA 6-2 \\
TU Berlin\\
Germany}
\email{\tt timmreck@math.tu-berlin.de}
\date{\today}

\begin{document}
\label{Timmreck:start}
\begin{abstract}
We associate with any simplicial complex $\K$ and any integer $m$ 
a system of linear equations and inequalities. If $\K$ has a
simplicial embedding in $\R^m$ then the system has an integer
solution.
This result extends the work of I.~Novik (2000). 
\end{abstract}

\maketitle

\section{Introduction}
In general, it is difficult to prove for
a simplicial complex $\K$ that it does \emph{not} have a simplicial embedding
(or not even a simplicial immersion) into $\R^m$. 

For example, the question whether
any neighborly simplicial surface 
on $n\ge12$ vertices can be realized in $\R^3$
leads to problems of this type.
Specifically, Amos Altshuler \cite{altshuler1996} has enumerated that
there are 59 combinatorial types of neighborly
simplicial $2$-manifolds of genus~$6$. Bokowski \& Guedes
de Oliveira \cite{bokowski2000} have employed oriented matroid enumeration methods
to show that one specific instance, number~54 from Altshuler's list,
does not have a simplicial embedding \index{simplicial embedding}; the other 58 cases were shown not to have 
simplicial embeddings only recently by L.~Schewe \cite{schewe2005}.

For \emph{piecewise linear} non-embeddability
proofs there is a classical set-up via obstruction classes, due to Shapiro \cite{shapiro1957} 
and Wu \cite{wu1965}.
In 2000, I. Novik \cite{novik2000}
has refined these obstructions for simplicial
embeddability:
She showed that if a \emph{simplicial embedding} of $\K$ in $\R^m$
exists, then a certain polytope in the cochain space
$C^m(\dprod{\K};\R)$ must contain an integral point.
Thus, infeasibility of a certain integer program 
might prove that a complex $K$ has no geometric realization.

In the following, we present Novik's approach (cf. parts \ref{mt:sym} and \ref{mt:coboundary}
 of Theorem \ref{thm:system}) 
in a reorganized way, 
so that we can work out more details, which allow us to sharpen some inequalities defining the
 polytope in $C^m(\dprod{\K};\R)$ (cf. Theorem \ref{thm:system}.\ref{mt:bounds}).
Further we interpret this polytope as a projection of a 
polytope in $C^m(\dprod{\S};\R)$, where $\S$ denotes the simplicial complex consisting of all faces of 
the $N$-simplex. The latter polytope is easier to analyze. This set-up is the right 
framework to work out the relations between variables (cf. Theorem \ref{thm:system}.\ref{mt:deformation}) 
and to express linking numbers (cf. Theorem \ref{thm:system}.\ref{mt:linking}), 
which are intersection numbers of cycles and empty simplices of $\K$ (which are present in $\S$ and
 therefore need no extra treatment.)
Using the extensions based on linking numbers we can show for 
a first example (Brehm's  triangulated M\"obius strip \cite{brehm1983})
that it is not simplicially embeddable in $\R^3$.

\section{A Quick Walk-Through}

Let $\K$ be a finite (abstract) simplicial complex
on the vertex set $V$, and fix a geometric realization $|\K|$ in some Euclidean space.
 % which we identify with $[n]:=\{1,2,\dots,n\}$,
Further let $f:V\to\R^m$ be any general position map
(that is, such that any $m+1$ points from~$V$
are mapped to affinely independent points in~$\R^m$).
Any such general position map extends affinely on 
every simplex to a \emph{simplicial map}\index{simplicial map}  $f: |\K| \to \R^m$ 
which we also denote by $f$. Such a simplicial map is a special case of a
piecewise linear map.

Every piecewise linear general position map $f$ defines an \emph{intersection cocycle}
\begin{equation}
  \label{eq:intersection_cochain}
\varphi_f\ \in\ C^m( \dprod{\K};\Z ).
\end{equation}
Here $\dprod{\K}$ denotes the \emph{deleted product}\index{deleted product} complex,
which consists of all faces $\sigma_1\times\sigma_2$
of the product $\K \times \K$ such that $\sigma_1$ and~$\sigma_2$ are disjoint
simplices (in~$\K$). As the deleted product is a polytopal complex we have the usual notions 
of homology and cohomology. For a detailed treatment of the deleted product complex we refer to 
\cite{matousek2003}.

The values of the intersection cocycle are given by
\[
\varphi_f(\sigma_1\times\sigma_2)\ =\ (-1)^{\dim \sigma_1}\I\big(f(\sigma_1),f(\sigma_2)\big),
\]
where $\I$ denotes the signed intersection number of the oriented simplicial 
chains $f(\sigma_1)$ and $f(\sigma_2)$ of complementary dimensions in~$\R^m$.
These intersection numbers (and thus the values of the
intersection cocycle) have the following key properties:
\begin{compactenum}[1.~]
  \item In the case of a simplicial map, all values
          $(-1)^{\dim \sigma_1}\I\big(f(\sigma_1),f(\sigma_2)\big)$ are $\pm1$ or~$0$.
          (In the greater generality of piece\-wise linear general position maps
          {$f:~\K\to~\R^m$,} as considered by Shapiro and  by Wu, $\I\big(f(\sigma_1),f(\sigma_2)\big)$ 
          is an integer.) 
  \item If $f$ is an embedding, then $\I\big(f(\sigma_1),f(\sigma_2)\big)=0$
          holds for any two disjoint simplices $\sigma_1,\sigma_2\in\K$.
  \item In the case of the ``cyclic map'' which maps $V$
          to the monomial curve of order $m$ (the ``moment curve''),
          the coefficients $(-1)^{\dim \sigma_1}\I\big(f(\sigma_1),f(\sigma_2)\big)$
          are given combinatorially. 
\end{compactenum}
The intersection cocycle is of interest since it defines
a cohomology class $\Phi_\K=[\varphi_f]$
that does not depend on the specific map $f$.
Thus, if some piecewise linear map $f$ is
an embedding, then $\Phi_\K$ is zero.

But a simplicial embedding is a special case of a piecewise linear embedding.
So the information $\Phi_\K$ is not 
strong enough to establish simplicial non-embeddability for complexes
that admit a piecewise linear embedding --- such as, for example, 
orientable closed surfaces in $\R^3$.

According to Novik we should therefore
study the specific coboundaries $\delta\lambda_{f,c}$ that establish
equivalence between different intersection cocycles.

So, \textbf{Novik's Ansatz} is to consider
\begin{equation}
  \label{eq:basic}
\fbox{$\displaystyle\quad\varphi_f-\varphi_c\ =\ \delta\lambda_{f,c}\quad$}
\end{equation}
where
\begin{compactitem}[$\bullet$~]
\item $\varphi_f\in C^m(\dprod{\K};\Z)$ is an integral
  vector, representing the intersection cocycle of a hypothetical
  embedding $f:\K\to\R^m$, so $\varphi_f \equiv 0$. (i.e. for every pair  $\sigma_1,\sigma_2 \in \K$
  of disjoint simplices, that $\varphi_f(\sigma_1\times\sigma_2)=0$),
\item $\varphi_c\in C^m(\dprod{\K};\Z)$ is an integral
  vector, whose coefficients $\varphi_c(\sigma_1\times\sigma_2)$ are known explicitly, representing
  the intersection cochain of the cyclic map $c:\K\to\R^m$,
\item $\delta$ is a known integral matrix with entries from $\{1,-1,0\}$
  that represents the coboundary map $\delta:C^{m-1}(\dprod{\K};\Z)\to C^m(\dprod{\K};\Z)$,
   and finally
\item  $\lambda_{f,c}\in C^{m-1}(\dprod{\K};\Z)$
  is an integral vector, representing the \emph{deformation co\-chain},
  whose coefficients are determined by $f$ and $c$, via  
\[
  \lambda_{f,c}(\tau_1\times\tau_2)\ =\ 
  \I\big( h_{f,c}(\tau_1\times I),h_{f,c}(\tau_2\times I)\big),
\]         
  where $h_{f,c}(x,t)=t f(x)+(1-t)c(x)$ interpolates between $f$
  and~$c$, for $t\in I:=[0,1]$.
\end{compactitem}
Thus if $\K$ has a simplicial embedding, then 
the linear system (\ref{eq:basic}) in the unknown vector
$\lambda_{f,g}$ has an \emph{integral} solution.
Moreover, Novik derived explicit bounds on the coefficients
of $\lambda_{f,g}$, that is, on the signed
intersection numbers between the parametrised
surfaces $h_{f,g}(\tau_1\times I)$ and $h_{f,g}(\tau_2\times I)$.

The intersection cocycles and deformation cochains induced by the 
general position maps $f,g: V \to \R^m$ on \emph{different} simplicial 
complexes $\K$ and $\tilde{\K}$ on the \emph{same} vertex set $V$ coincide on $\dprod{\K} 
\cap \dprod{\tilde{\K}}$. They are projections of the same intersection cocycle
or deformation cochain on $\dprod{\S}$, where $\S$ denotes the full face 
lattice of the simplex with vertex set $V$. 
We therefore investigate these largest cochains and get Novik's results back as well 
as some stronger results even in the original setting;
see Theorem~\ref{thm:properties} and Remark~\ref{remark:novik}.

In the following, we 
\begin{compactitem}[$\bullet$~]
  \item derive the validity of the basic equation (\ref{eq:basic}), in
    Section~\ref{section:obstruction_theory},
  \item examine deformation cochains induced by general position maps on the vertex set in Section~\ref{section:pl_vs_geo}, and
  \item exhibit an obstruction system to geometric realizability in Section~\ref{section:extensions}.
\end{compactitem}
Furthermore, in Section~\ref{section:experiments} we discuss subsystems and report about computational results.

\section{Obstruction Theory}
\label{section:obstruction_theory}
We state and prove the results of this section for
\emph{simplicial} maps only. They hold in the more
general framework of \emph{piecewise linear} maps as well. For proofs and 
further details in this general setting we refer to Wu
\cite{wu1965}.

\subsection{Intersections of Simplices and Simplicial Chains}
\label{section:intersection_simplex}

\begin{definition}
Let $\sigma$ and $\tau$ be affine simplices of complementary dimensions $k+\ell
= m$ in $\R^m$ 
with vertices $\sigma_0, \ldots, \sigma_k$ and
$\tau_0, \ldots, \tau_\ell$ respectively, of complementary dimensions $k+\ell
= m$. Suppose  that $\sigma_0, \ldots, \sigma_k, \tau_0, \ldots, \tau_\ell$ are in
general position and the simplices are oriented according to the
increasing order of the indices. Then $\sigma$ and $\tau$ intersect in
at most one point. The \emph{intersection number}\index{intersection number!of simplices} $\I\big(\sigma,
\tau\big)$ is defined to be zero if $\sigma$ and $\tau$ don't intersect
and $\pm 1$ according to the orientation of the full dimensional
simplex $(p, \sigma_1, \ldots, \sigma_k, \tau_1, \ldots, \tau_\ell)$
if $\sigma$ and $\tau$ intersect in $p$. This definition extends
bilinearly to simplicial chains in $\R^m$. (We consider integral chains, that is, 
formal combinations of affine simplices in $\R^m$ with integer coefficients.)
\end{definition}
\begin{lemma}
Let $x$, $y$ be simplicial chains in $\R^m$ with $\dim x = k$ and
$\dim~y~=~\ell$.
\begin{enumerate}[\rm (a)]
\item If $k+\ell = m$ then $\I(x, y) = (-1)^{k\ell} \I(y, x)$.
\item If $k+\ell = m+1$ then $\I(\partial x, y) = (-1)^{k} \I(x, \partial y)$.
\end{enumerate}
\end{lemma}

Now we use intersection numbers to associate a cocycle to each general
position map.

\begin{lemdef}
Let $f: \langle N \rangle \to \R^m$ be a general position map.
The cochain
defined by
 \[
\varphi_f (\sigma_1 \times \sigma_2) := (-1)^{\dim \sigma_1} \I
\big(f(\sigma_1), f(\sigma_2)\big) \qquad \mbox{for } m \mbox{-cells }  \sigma_1
\times \sigma_2 \in \dprod{\K} 
\]
is a cocycle. It is called the 
\emph{intersection cocycle} of $f$.

\end{lemdef}
The intersection cocycle has the following symmetries. For every
$m$-cell $\sigma_1 \times \sigma_2$, with $\dim \sigma_1 = k$ and
$\dim \sigma_2 = \ell$,
\begin{equation*}
\varphi_f (\sigma_1 \times \sigma_2) 
= (-1)^{(k+1)(\ell+1)+1}\varphi_f (\sigma_2 \times \sigma_1).
\end{equation*}

\begin{remark} Wu calls this cocycle \emph{imbedding cocyle} \cite[p.183]{wu1965}. 
If $f$ is a piecewise linear embedding, then $\varphi_f = 0$.
When we look at simplicial maps we even have an equivalence: 
A simplicial map $f$ is an embedding of $\K$ if and only if the
intersection cocycle is $0$. So $\varphi_f$ measures the
deviation of $f$ from a geometric realization. This makes the intersection 
cocycle quite powerful.
\end{remark}

\subsection{Intersections of Parametrized Surfaces}
\label{section:intersection_surface}
In this section we sort out definitions, fix orientations 
and establish the fundamental relation in Proposition \ref{prop:fundamental} (cf. \cite[pp. 180 and 183]{wu1965}).
Wu uses a simplicial homology between two different piecewise linear maps to establish 
the independence of the homology class of the particular piecewise linear map. We use  
a straight line homotopy instead. 
  
\begin{definition}
Let $U \subset \R^k$ and $V \subset \R^\ell$ be sets that are closures of their interiors 
and $\varphi : U \to \R^m$ and $\psi : V \to \R^m$ smooth parametrized surfaces.
The surfaces $\varphi$ and $\psi$ \emph{intersect transversally} at $p =
\varphi(\alpha) = \psi (\beta)$ with $\alpha \in \overset{\circ}{U}$ and $\beta
\in \overset{\circ}{V} $, if 
\[
T_p\R^m = d\varphi(T_\alpha U) \oplus d\psi(T_\beta V)
\]
In other words $k + \ell = m$ and the vectors 
\[
\left.\frac{\partial \varphi}{\partial u_1}\right|_\alpha , \ldots, 
\left.\frac{\partial \varphi}{\partial u_k}\right|_\alpha ,
\left.\frac{\partial \psi}{\partial v_1}\right|_\beta , \ldots, 
\left.\frac{\partial \psi}{\partial v_\ell}\right|_\beta 
\]
span $\R^m$.
In this situation the index of intersection of $\varphi$ and $\psi$ in
$p$ is defined by 
\[
\I_p(\varphi, \psi) \ \ := \ \  \sgn \det \left(
\left.\frac{\partial \varphi}{\partial u_1}\right|_\alpha , \ldots, 
\left.\frac{\partial \varphi}{\partial u_k}\right|_\alpha ,
\left.\frac{\partial \psi}{\partial v_1}\right|_\beta , \ldots, 
\left.\frac{\partial \psi}{\partial v_\ell}\right|_\beta
\right). 
\]
The surfaces $\varphi$ and $\psi$ are in \emph{general position} if they
intersect transversally only. In particular there are no intersections
at the boundary. Surfaces in general position intersect in finitely many points only and the
\emph{intersection number} \index{intersection number!of parametrized surfaces} is defined by
\[
\I(\varphi, \psi)\ \ 
:= \sum_{p = \varphi(\alpha) = \psi(\beta)} \I_p(\varphi, \psi).
\]
We also write $\I\big(\varphi(U), \psi(V)\big)$ for $\I(\varphi, \psi)$ when we want to 
emphasize the fact that the images intersect.
\end{definition}

We now give parametrizations of simplices so that the two definitions coincide.
\begin{notation*}
Denote by $(e_1, \ldots, e_m)$ the standard basis of $\R^m$ and let
$e_0 := \zerovec$. Further let  $[m] := \{1, \ldots, m\}$, $\langle m
\rangle := [m] \cup \{0\}$ and for $I \subseteq \langle m \rangle$ let
$\simp_I$ denote the simplex $\conv\{e_i\mid i\in I\}$.
Finally let $J = \{ j_0, \ldots, j_k\}_<$ denote the set $\{ j_0, \ldots, j_k\}$ with $j_0<\ldots< j_k$.
\end{notation*}
For a simplex $\sigma = \conv\{\sigma_0, \ldots , \sigma_k\}$ the parametrization
$\varphi_\sigma : \R^k \supset \Delta_{[k]} \to \sigma \subset \R^m,
(u_1, \ldots, u_k) \mapsto \sigma_0 + \sum_{i=1}^k u_i (\sigma_i -
\sigma_0)$ induces the orientation corresponding to the increasing 
order of the indices.
Now consider two simplices $\sigma = \conv\{\sigma_0, \ldots , \sigma_k\}$ and 
$\tau = \conv\{\tau_0, \ldots , \tau_\ell\}$.
If $\{\sigma_0, \ldots , \sigma_k,\tau_0, \ldots , \tau_\ell\}$ is 
in general position then also $\varphi_\sigma$ and $\varphi_\tau$ are
in general position. Let $\sigma$ and $\tau$ intersect in 
\begin{eqnarray*}
p &=& \sum_{i=0}^k \alpha_i \sigma_i =
\sigma_0 + \sum_{i=1}^k \alpha_i (\sigma_i - \sigma_0) =
\varphi_\sigma(\alpha) \\
&=&  \sum_{i=0}^\ell \beta_i \tau_i = 
\tau_0 + \sum_{i=1}^\ell \beta_i (\tau_i -\tau_0)=
\varphi_\tau(\beta).
\end{eqnarray*}
Then we have by a straightforward calculation:
\begin{eqnarray*}
\I(\sigma, \tau) &=& \sgn \det \left( \binom{1}{p}, 
                              \binom{1}{\sigma_1}, \ldots, \binom{1}{\sigma_k}, 
                              \binom{1}{\tau_1}, \ldots, \binom{1}{\tau_\ell} \right)\\        
&=& \sgn \det \left(
\left.\frac{\partial \varphi_\sigma}{\partial u_1}\right|_\alpha , \ldots, 
\left.\frac{\partial \varphi_\sigma}{\partial u_k}\right|_\alpha ,
\left.\frac{\partial \varphi_\tau}{\partial v_1}\right|_\beta , \ldots, 
\left.\frac{\partial \varphi_\tau}{\partial v_\ell}\right|_\beta
\right)\\
&=& \I(\varphi_\sigma, \varphi_\tau)\ .
\end{eqnarray*}

In the following we use the parametrization
$\varphi_{|J|} \times \id$ that induces the product orientation on  $|J|\times \R$.

\begin{definition}
\label{lambda}
Let $f, g : \langle N \rangle \to \R^m$ be two general position maps such that
$\{f(i): i \in \langle N \rangle \} \cup \{g(i): i \in \langle N \rangle \}$ 
is in general position, where 
$f(i) = g(j)$ is permitted only if $i = j$.
Define the
\emph{deformation map}
\[h_{f,g} :\  |\K| \times \R \ \ \to \ \ \R^m \times \R
\]
\[ 
h_{f,g} (x,t)\ \  :=\ \  (tf(x)+(1-t)g(x), t).
\]
and the 
\emph{deformation cochain}\index{deformation cochain} $\lambda_{f,g} \in \mathcal{C}^{m-1}
(\dprod{\K})$ of $f$  and $g$ by
\[
\lambda_{f,g} (\tau_1 \times
\tau_2) \ \ := \ \ \I \big(h_{f,g} (|\tau_1| \times [0,1]), h_{f,g}
(|\tau_2| \times [0,1])\big) \]
for $(m-1)$-cells $\tau_1
\times \tau_2 \in \dprod{\K}$.

\end{definition}
\begin{prop}
\label{prop:fundamental}
The cohomology class of $\varphi_f$ is independent of the general position map $f$:
For two general position maps $f$ and $g$ we have 
\[
\delta \lambda_{f,g} \ \ = \ \ \varphi_f - \varphi_g.
\]
Therefore the cohomology class $\Phi_\K 
:= [\varphi_f] \in H^m(\dprod{\K};\Z)$ is an invariant of the complex $\K$ itself.
\end{prop}

\begin{proof}
Let $\sigma \times \tau \in \dprod{\K}$, $\dim \sigma
\times \tau = m$. In the following we omit the index $f,g$ from
$\lambda_{f,g}$ and $h_{f,g}$.
We get the boundary of $h(\sigma \times [0,1])$ by taking the boundary
first and then applying $h$. The intersections $h(\partial \sigma
\times [0,1]) \cap h(\tau \times [0,1])$ are inner intersections. We
extend the surface patch $h(\tau \times [0,1])$ to $h(\tau \times
[-\varepsilon,1+\varepsilon])$ so that the intersections $h(\sigma
\times \{0\}) \cap h(\tau \times \{0\})$ and $h(\sigma
\times \{1\}) \cap h(\tau \times \{1\})$ become inner intersections of
$h(\partial (\sigma \times [0,1])) \cap h(\tau \times [-\varepsilon,1+\varepsilon])$
as well but no new intersections occur.
Then we have
\begin{eqnarray*}
\lambda(\partial \sigma \times \tau) &=& \I\big(h(\partial \sigma
\times [0,1]), h(\tau \times [0,1])\big)\\
&=& \I\big(h(\partial \sigma
\times [0,1]), h(\tau \times [-\varepsilon,1+\varepsilon])\big)\\
&=& \I\big(h(\partial (\sigma \times [0,1])), h(\tau \times
[-\varepsilon,1+\varepsilon])\big)\\
&&+ (-1)^{\dim \sigma} \I\big(h(\sigma \times \{0\}), h(\tau \times
[-\varepsilon,1+\varepsilon])\big)\\
&&- (-1)^{\dim \sigma} \I\big(h(\sigma \times \{1\}), h(\tau \times
[-\varepsilon,1+\varepsilon])\big)\\
&=& (-1)^{\dim \sigma +1} \I\big(h(\sigma \times [0,1]), h(\partial(\tau \times
[-\varepsilon,1+\varepsilon])\big)\\
&&+ (-1)^{\dim \sigma} \I\big(f(\sigma), f(\tau)\big)
- (-1)^{\dim \sigma} \I\big(g(\sigma), g(\tau)\big)\\
&=&  (-1)^{\dim \sigma +1} \big[\I\big(h(\sigma \times [0,1]), h(\partial\tau \times
[-\varepsilon,1+\varepsilon])\big)\\
&&\phantom{(-1)^{\dim \sigma +1} \big[} + \I\big(h(\sigma \times [0,1]), h(\tau \times
\{-\varepsilon\})\big)\\
&&\phantom{(-1)^{\dim \sigma +1} \big[} - \I\big(h(\sigma \times [0,1]), h(\tau \times
\{1+\varepsilon\}\})\big)\big]\\
&&+ (-1)^{\dim \sigma} \I\big(f(\sigma), f(\tau)\big)
- (-1)^{\dim \sigma} \I\big(g(\sigma), g(\tau)\big)\\
&=& (-1)^{\dim \sigma +1} \lambda (\sigma \times \partial \tau)
+\varphi_f(\sigma \times \tau) -\varphi_g(\sigma \times \tau) 
\end{eqnarray*}
\end{proof}

The deformation cochain has symmetries as well:
\begin{lemma}
If $\tau_1 \times \tau_2$ is an $(m-1)$-cell of 
$\dprod{\K}$ then $\tau_2 \times \tau_1$ is also an $(m-1)$-cell of 
$\dprod{\K}$ and 
\label{lem:symmetries}
\[
\lambda_{f,g} (\tau_1 \times \tau_2) 
\ \ = \ \ (-1)^{(\dim \tau_1 +1)(\dim \tau_2 +1)} \lambda_{f,g} (\tau_2 \times \tau_1).
\]
\end{lemma}
\begin{proof}
\begin{eqnarray*}
\lambda_{f,g} (\tau_1 \times \tau_2) &=&
    \I \big(h_{f,g} (\tau_1 \times [0,1]), h_{f,g}(\tau_2 \times
    [0,1])\big)\\
&=& (-1)^{(\dim \tau_1 +1)(\dim \tau_2 +1)} \I \big(h_{f,g} 
    (\tau_2 \times [0,1]), h_{f,g}(\tau_1 \times [0,1])\big)\\
&=& (-1)^{(\dim \tau_1 +1)(\dim \tau_2 +1)} \lambda_{f,g} (\tau_2 \times \tau_1)
\end{eqnarray*}
\end{proof}
\begin{remark} Intersection cocycle and deformation cochain can also
be defined for \emph{piecewise linear} general position maps maintaining the
same properties \cite{wu1965}. So the
cohomology class $\Phi_{\K} := [\varphi_f]\in H^m(\dprod{\K})$ where $f : |\K| \to \R^m$ is any piecewise linear
map, only serves as an obstruction to piecewise
linear embeddability. It cannot distinguish between piecewise linear
embeddability and geometric realizability.
\end{remark}

\section{Distinguishing between Simplicial Maps and P.L. Maps}
\label{section:pl_vs_geo}

In this paragraph  we collect properties of deformation cochains between \emph{simplicial} maps that 
do not necessarily hold for deformation cochains between arbitrary \emph{piecewise linear} maps.
The values of the intersection cocycles $\varphi_f$, $\varphi_g$ and the deformation cochain $\lambda_{f,g}$ of two 
simplicial maps $f$ and $g$ depend only on the values that $f$ and $g$ take on the vertex 
set $\langle N\rangle$ of the complex in question. The complex itself  
determines the products $\sigma \times \tau$ on which $\lambda_{f,g}$ may be evaluated. 
So we examine what values these cochains take on
$\mathcal{C}^{m-1}(\dprod{\S})$, where $\S$ denotes the
$m$-skeleton of the $N$-simplex. In Section \ref{section:extensions} we derive further properties for the case
that we deform into a geometric realization.

\subsection{Linking Numbers}

\begin{definition}
Let $x$, $y$ be simplicial cycles in $\R^m$, $\dim x+ \dim y = m+1$, 
with disjoint supports. As every cycle bounds in $\R^m$ we find a chain $\gamma$
such that $\partial \gamma = x$. The \emph{linking number}\index{linking number} of $x$ and $y$ is 
defined as $\L(x,y) := \I(\gamma, y)$.
\end{definition}

\begin{lemma}
\label{lemma:bounds}
Let $\sigma$, $\tau$ be affine simplices in $\R^m$. Then:
\begin{enumerate}[\rm (a)]
\item $|\I(\sigma, \tau)| \  \leq \  1$ 
      if $\dim \sigma + \dim \tau = m$.   
\item $|\L(\del \sigma, \del \tau)| \ \leq \ 1$ if $\dim
  \sigma = 2$ and $\dim \tau = m-1$.
\end{enumerate}
\end{lemma}

The next two conditions follow from the estimates in Lemma \ref{lemma:bounds}
on the intersection numbers $\varphi_f$.

\begin{prop}
\label{prop:intersection_inequalities}
Let $f$, $g$ be two general position maps of $\langle N \rangle$ into
$\R^m$.\\
Then 
\begin{enumerate}[\rm (a)]
\item $-1 -\varphi_g(\sigma \times \tau) \ \ \leq \ \  \delta
  \lambda_{f,g}(\sigma \times \tau) \ \ \leq\ \ 1 -\varphi_g(\sigma
  \times \tau)$ \\
for all $\sigma \times \tau \in
\dprod{\S}$, $\dim \sigma +\dim \tau = m$. 
\item $-1 -\varphi_g(\sigma \times \del \tau) \ \ \leq \ \ 
  \lambda_{f,g}(\del \sigma \times \del \tau) \ \ \leq\ \ 1 -\varphi_g(\sigma
  \times \del \tau)$ \\
for all $\sigma \times \tau \in
 \dprod{\S}$, $\dim \sigma = m-1$ and $\dim \tau = 2$. 
\end{enumerate}
\end{prop}

\begin{proof}
  \begin{enumerate}[\rm (a)]
  \item  As $\varphi_f(\sigma \times \tau) = (-1)^{\dim \sigma}
         \I\big(f(\sigma), f(\tau)\big)$ we can bound $\varphi_f$ in Proposition \ref{prop:fundamental} by $|\varphi_f| \leq 1$.
  \item $\varphi_f(\sigma \times \del\tau) = (-1)^{\dim \sigma}\I\big(f(\sigma), 
        f(\del\tau)\big) = (-1)^{\dim \sigma}\L\big(f(\del\sigma), f(\del\tau)\big)$ and \\
        $\delta
        \lambda_{f,g}(\sigma \times \del\tau) = \lambda_{f,g}(\del\sigma 
        \times \del\tau)$.
  \end{enumerate}
\end{proof}

\subsection{Deforming Simplices}

\subsubsection{The Simplest Case}
\label{section:deformation}
For the following, homotopies between images of a simplicial complex under 
different general position maps play a crucial r\^ole.
In this section we look at the simplest case: The
homotopy from the standard simplex of $\R^m$ to an arbitrary one.

Let $D \in \R^{m \times m}$ be an arbitrary matrix with columns
$d_i$, $i \in [m]$ and set $d_0 := \zerovec$. Associate with $D$ the
map 
\[
h: \R^{m+1} \to \R^{m+1}, h(x, t):= ((tD + (1-t)E_m) x, t).
\]
Then for every subset $I \subset \langle m \rangle$ the map $h|_{\simp_I \times [0,1]}$ represents the homotopy of $\simp_I$
into $\conv\{d_i\mid i \in I\}$,  moving all
points along straight line segments to corresponding points, i.e. it
is a ruled $m$-dimensional surface.
\vspace{0.5cm}
\begin{figure}
\input{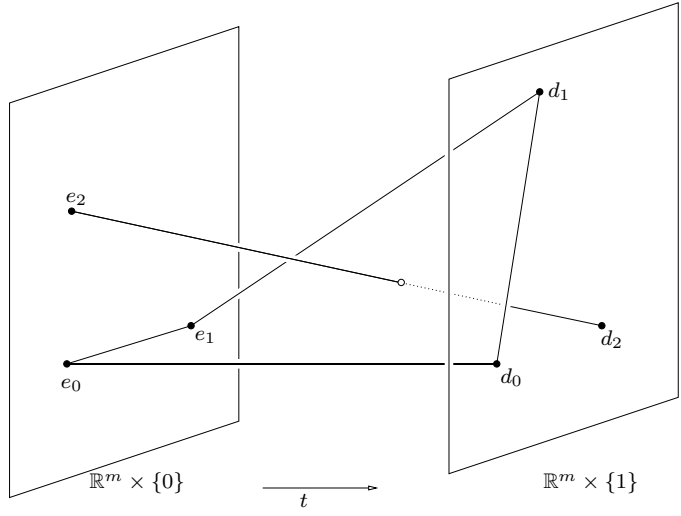}
\caption{Intersecting surfaces $h(\simp_{I_+} \times \R)$
  and $h(\simp_{I_-} \times \R)$ for the matrix $(d_1, d_2)$ and the 
partition $I_+ = \{0,1\}$ and $I_- = \{2\}$}
\label{figure:surfaces}
\end{figure}

We call an eigenvalue of a square matrix \emph{general} if it is simple, its
eigenvector $v$ has no vanishing components, and $\sum v_i \neq 0$. 
This technical condition characterizes the situation where all pairs of ruled 
surfaces defined by disjoint subsets of the vertex set
are transversal.

We begin by characterizing intersection points of pairs of surfaces in terms of eigenvalues of $D$.
\begin{lemma}
\label{lemma1}
Let $D \in \R^{m \times m}$ and $h: \R^{m+1} \to \R^{m+1}$ its
associated map.
\begin{enumerate}[\rm (a)]
\item  
Let $I_+, I_- \subset \langle m\rangle$ such that $I_+ \cap I_- = \emptyset$.
If the surfaces $h(\simp_{I_+} \times \R)$
  and $h(\simp_{I_-} \times \R)$ intersect at time $t$, then $1 -
  \frac{1}{t}$ is an eigenvalue of $D$.
\item Let $u \neq 1$ be a general eigenvalue of $D$.
  Then $u$ uniquely determines disjoint subsets $I_+^u$ and $I_-^u \subset
  \langle m \rangle$ with $0 \in I_+^u$ such that 
  $h(\simp_{I_+^u} \times \R)$ and $h(\simp_{I_-^u} \times \R)$
  intersect at time \[ t = \tfrac{1}{1-u}.\] 
\end{enumerate} 
\end{lemma}
\emph{Another point of view:} If $u \neq 1$ is an eigenvalue of $D$ then $h(\simp_{\langle m\rangle}\times \{t\})$ fails to span $\R^m \times \{t\}$. So we get a Radon partition in some lower dimensional subspace of $\R^m \times \{t\}$. If the eigenvector is general then we get a unique Radon partition.
\begin{proof}
\begin{enumerate}[\rm (a)]
\item Let $(p, t) \in h(\simp_{I_+} \times \R)
  \cap h(\simp_{I_-} \times \R) $ be an intersection point. Then $p$
  has the representation 
\[
p \ \ = h(\alpha, t) = h(\beta, t),
\] 
that is,
\[
p \ \ = \ \ \sum_{i \in I_+} \alpha_i (td_i + (1-t)e_i) \  = \   \sum_{j \in I_-} \beta_j (td_j + (1-t)e_j) 
\]  
with $$\sum_{i \in I_+} \alpha_i = \sum_{j \in I_-} \beta_j = 1,$$
$\alpha_i, \beta_j >0$ for all $i \in I_+$, $j \in I_-$.
Because of $t \neq 0$ and $e_0 = d_0 = \zerovec$  we can rewrite this as
\[
t\ \Big( \sum_{i \in I_+} \alpha_i \big(d_i -(1 - \tfrac{1}{t})e_i\big) +
  \sum_{j \in I_-} (-\beta_j) \big(d_j -(1 - \tfrac{1}{t})e_j\big) \Big) \ = \ \zerovec  
\]
Therefore $\sum_{i \in I_+} \alpha_i e_i + \sum_{j \in I_-} (-\beta_j)
e_j$ is an eigenvector of $D$ with eigenvalue $1- \frac{1}{t}$.
\item
Let $u \neq 1$ be a general eigenvalue of $D$ and $v$ its eigenvector. Consider the sets
$\tilde{I}^u_+ := \{i \in [m] \mid  v_i >0\}$ and  $I^u_- := 
\{i \in [m] \mid  v_i < 0\}$ of positive and negative coefficients respectively. Without loss of generality
assume that $V := -\sum_{i \in I^u_-} v_i > \sum_{i \in \tilde{I}^u_+}
v_i$.
Denote $I^u_+ := \tilde{I}^u_+ \cup \{0\}$, $\alpha_i:= \frac{v_i}{V}$ 
for $i \in \tilde{I}^u_+$ and $\alpha_0 := 1 -\sum_{i \in
  \tilde{I}^u_+} \alpha_i$ ,  $\beta_j := -\frac{v_j}{V}$ for $j \in
I_-$ and $t := \frac{1}{1-u}$. Then
$(p, t)$ with
\[
p \ \ = \ \ \sum_{i \in I^u_+} \alpha_i (td_i + (1-t)e_i) \  = \   \sum_{j \in I^u_-} \beta_j (td_j + (1-t)e_j) 
\] is an intersection  point of the two simplices $h(\simp_{I^u_+} \times \{t\})$ and
$h(\simp_{I^u_-} \times \{t\})$.\\
So the surfaces $h(\simp_{I^u_+} \times
\R)$ and $h(\simp_{I^u_-} \times \R)$ intersect at time $t$. 
\end{enumerate}
\end{proof}

\begin{remark}
In the case of a general eigenvalue  $u = 1$ we can still find the
sets $I_+^u$ and $I_-^u$. Then the surfaces $h(\simp_{I_+^u} \times \R)$ and $h(\simp_{I_-^u} \times \R)$
have parallel ends. This complements the preceding
lemma because they then \lq meet at time $t = \infty$\rq .
\end{remark}

Denote by  
\[\mathcal{P} := \{\{I_+, I_-\} \mid  I_+ \cup I_- =
\langle m \rangle, I_+ \cap  I_- = \emptyset , 0\in I_+\}
\] 
the set of all bipartitions of $\langle m \rangle$.
\begin{cor}
\label{cor:points}
Let $D \in \R^{m \times m}$ and $\ell$ be the multiplicity of the 
eigenvalue $1$ of $D$.\\
 Then
\[
\sum_{\{I_+, I_-\} \in \mathcal{P}} \# \left( h(\simp_{I_+} \times \R) \cap
h(\simp_{I_-} \times \R)\right) \ \ 
 \leq \ \ m - \ell \ ,
\]
that is, the total number of intersection points of pairs of surfaces of the form
$h(\simp_I \times \R)$ and $h(\simp_{\langle m \rangle \setminus I} \times \R)$ can
not exceed $m-\ell$.
\end{cor}
Now we calculate intersection numbers of the surfaces found in
Lemma \ref{lemma1}. To this end we impose orientations on the surfaces
in question.
In the following let $h(\simp_{I} \times \R)$ carry the orientation
induced by the parametrization $\psi := h \circ (\varphi_{I} \times
\id)$. Further let $I_+^u= \{i_0, \ldots , i_k \}_<$, $I_-^u= \{i_{k+1}, \ldots , i_m
\}_<$ with $i_0  = 0$
and denote by
$(I^u_+, I^u_-)$ the \emph{\lq shuffle\rq \ permutation} $(i_0, \ldots , i_m) \mapsto (0, \ldots , m)$.\\
\begin{lemma}
\label{lemma:index}
Let $D \in \R^{m \times m}$ and $u = 1 - \frac{1}{t}$ be an eigenvalue
of $D$. Denote by $(p, t)$ the intersection point of 
the surfaces $h(\simp_{I_+^u} \times \R)$ and $h(\simp_{I_-^u} \times \R)$.
The surfaces 
intersect transversally and  $(p, t)$ is an inner point  if and only if $u$ is general.
In this case we have
\[
\left. \I\big(h(\simp_{I_+^u} \times \R), h(\simp_{I_-^u} \times
\R)\big)\right| _{(p, t)} = \sgn (I^u_+, I^u_-)\sgn (t^m \chi^\prime_D(u)),
\]
where $\chi^\prime_D$ is the derivative of the characteristic
polynomial $\chi_D(u) = \det(D-uE_m)$ of $D$.
\end{lemma}
\begin{proof}
Our calculations differ in so far from those in \cite[Proof of Lemma 3.2]{novik2000} as we
have to deal with the permutation $(I^u_+, I^u_-) : j \mapsto i_j$.
Denote the intersection point of the surfaces in question by
$(p,t)$ where
\begin{equation}
\label{p}
p \ \ =\ \  t\sum_{j=0}^k \alpha_{i_j} d_{i_j}^u = 
t\sum_{j = k+1}^m \beta_{i_j} d_{i_j}^u
\end{equation}
with $d_i^u := d_i - u e_i$.

For checking transversality as well as for the index of intersection at
$(p,t)$ we examine
\[
\mathcal{D}:= \det \left( \left(\frac{\del
      \psi_+}{\del \xi_1 }\right), \ldots,\left(\frac{\del
      \psi_+}{\del\xi_k}\right), \left(\frac{\del
      \psi_+}{\del t}\right), \left( \frac{\del
      \psi_-}{\del \xi_1 }\right), \ldots,\left(\frac{\del
      \psi_-}{\del\xi_{m-k}}\right), \left(\frac{\del
      \psi_-}{\del t}\right)\right),
\]
where the first $k$
derivatives are calculated at $(\alpha_{i_1}, \ldots
,\alpha_{i_k}, t)$ and the last $m-k$ at $(\beta_{i_{k+2}}, \ldots
,\beta_{i_{m}}, t)$. We therefore get
\begin{eqnarray*}
\mathcal{D} &=&
 \lefteqn{\det \left( t \binom{d_{i_1}^u}{0}, \ldots, t \binom{d_{i_k}^u}{0}, 
    \binom{\sum_{j=1}^{k} \alpha_{i_j} (d_{i_j}\!\!\! -e_{i_j})}{1}, \right.}\\
 & & \left. \phantom{det~(~}  t \binom{d_{i_{k+2}}^u\!\!\! - d_{i_{k+1}}^u}{0}, \ldots
  , t \binom{d_{i_m}^u\!\!\! - d_{i_{k+1}}^u}{0}, 
  \binom{\sum_{j=k+1}^{m} \beta_{i_j} (d_{i_j}\!\!\! -e_{i_j})}{1}  \right).
\end{eqnarray*}
With $i_0 = 0$, $d_0 = e_0 = \zerovec$ we have
\[\binom{\sum_{j=1}^{k} \alpha_{i_j} (d_{i_j}\!\!\! -e_{i_j})}{1}  -  
  \binom{\sum_{j=k+1}^{m} \beta_{i_j} (d_{i_j}\!\!\! -e_{i_j})}{1} \ \ 
  =  \ \ \tfrac{1}{t} v\ ,
\]
as $v =\sum_{j=1}^{k} \alpha_{i_j} e_{i_j} - \sum_{j=k+1}^{m} \beta_{i_j}e_{i_j}$ is also an eigenvector of $D-E$ with eigenvalue $\frac{1}{t}$.
Subtracting the last column from the $(k+1)$st and 
using Laplace expansion with respect to the last row we get 
\begin{eqnarray}
\nonumber
\mathcal{D} &=& t^{m-2} \det (d^u_{i_1}, \ldots, d^u_{i_k}, v,
d^u_{i_{k+2}}\!\!\!- d^u_{i_{k+1}}, \ldots,  d^u_{i_m}\!\!\!- d^u_{i_{k+1}})\\
&=&
\label{d} t^{m-2} \sum_{j=1}^m v_j \det (d^u_{i_1}, \ldots, d^u_{i_k}, e_j,
d^u_{i_{k+2}}\!\!\!- d^u_{i_{k+1}}, \ldots,  d^u_{i_m}\!\!\!- d^u_{i_{k+1}} )
\end{eqnarray}
From \eqref{p} we have
\begin{equation*}
0 \ \ = \ \ \sum_{j = 1}^k \alpha_{i_j} d_{i_j}^u \ \  -\  d^u_{i_{k+1}} \ -\sum_{j =
    k+2}^m \beta_{i_j} (d_{i_j}^u\!\!\! - d^u_{i_{k+1}})\ .
\end{equation*}
Now we examine the summands of the last expression of $\mathcal{D}$ in
three groups.
In the first case, $j < k+1$, we have $v_{i_j} = \alpha_{i_j}$. We substitute $\alpha_{i_j}d^u_{i_j}$, cancel all terms except
$d^u_{i_{k+1}}$ and exchange $d^u_{i_{k+1}}$ and $e_{i_j}$:
\[
\begin{split}
\alpha_{i_j}& \det (d^u_{i_1}, \ldots, d^u_{i_k}, e_{i_j},
d^u_{i_{k+2}}\!\!\!- d^u_{i_{k+1}}, \ldots,  d^u_{i_m}\!\!\!- d^u_{i_{k+1}} )\\
&= \det (d^u_{i_1}, \ldots, \alpha_{i_j}d^u_{i_j}, \ldots, d^u_{i_k}, e_{i_j},
d^u_{i_{k+2}}\!\!\!- d^u_{i_{k+1}}, \ldots,  d^u_{i_m}\!\!\!- d^u_{i_{k+1}} )\\
&=\det (d^u_{i_1}, \ldots, d^u_{i_{k+1}}, \ldots, d^u_{i_k}, e_{i_j},
    d^u_{i_{k+2}}\!\!\!- d^u_{i_{k+1}}, \ldots,  d^u_{i_m}\!\!\!- d^u_{i_{k+1}})\\
&=-\det (d^u_{i_1}, \ldots,e_{i_j} , \ldots,  d^u_{i_m})\ .
\end{split}
\]
By an analogous calculation the second case, $j > k+1$,  yields
\[
\begin{split}
-\beta_{i_j}& \det (d^u_{i_1}, \ldots, d^u_{i_k}, e_{i_j},
d^u_{i_{k+2}}\!\!\!- d^u_{i_{k+1}}, \ldots,  d^u_{i_m}\!\!\!- d^u_{i_{k+1}} )\\
&=-\det (d^u_{i_1}, \ldots, e_{i_j} , \ldots,  d^u_{i_m})\ .
\end{split}
\]
For the remaining term, $j=k$,  we use the same procedure on each of the summands 
after the first step and evaluate the telescope sum in the last step. 
Thus we get:
\[
\begin{split}
&-\beta_{i_{k+1}} \det (d^u_{i_1}, \ldots, d^u_{i_k}, e_{i_{k+1}},
d^u_{i_{k+2}}\!\!\!- d^u_{i_{k+1}}, \ldots,  d^u_{i_m}\!\!\!- d^u_{i_{k+1}} )\\
&= -\det (d^u_{i_1}, \ldots, d^u_{i_k}, e_{i_{k+1}},
d^u_{i_{k+2}}\!\!\!- d^u_{i_{k+1}}, \ldots,  d^u_{i_m}\!\!\!- d^u_{i_{k+1}} )\\
&\quad + \sum_{j=k+2}^m  \det (d^u_{i_1}, \ldots, d^u_{i_k}, e_{i_{k+1}},
d^u_{i_{k+2}}\!\!\!- d^u_{i_{k+1}}, \ldots,  \beta_{i_j}(d^u_{i_j}\!\!\!-
d^u_{i_{k+1}}), \ldots,  d^u_{i_m}\!\!\!- d^u_{i_{k+1}} )\\
&= -\det (d^u_{i_1}, \ldots, d^u_{i_k}, e_{i_{k+1}},
d^u_{i_{k+2}}\!\!\!- d^u_{i_{k+1}}, \ldots,  d^u_{i_m}\!\!\!- d^u_{i_{k+1}} )\\
&\quad + \sum_{j=k+2}^m  \det (d^u_{i_1}, \ldots, d^u_{i_k}, e_{i_{k+1}},
d^u_{i_{k+2}}\!\!\!- d^u_{i_{k+1}}, \ldots,  -d^u_{i_{k+1}}, \ldots,  d^u_{i_m}\!\!\!- d^u_{i_{k+1}} )\\
&= -\det (d^u_{i_1}, \ldots, d^u_{i_k}, e_{i_{k+1}},
d^u_{i_{k+2}}\!\!\!- d^u_{i_{k+1}}, \ldots,  d^u_{i_m}\!\!\!- d^u_{i_{k+1}} )\\
&\quad - \sum_{j=k+2}^m  \det (d^u_{i_1}, \ldots, d^u_{i_k}, e_{i_{k+1}},
d^u_{i_{k+2}}, \ldots,d^u_{i_{j-1}},  d^u_{i_{k+1}}, d^u_{i_{j+1}}\!\!\!- d^u_{i_{k+1}}\ldots,  d^u_{i_m}\!\!\!- d^u_{i_{k+1}} )\\
&=-\det (d^u_{i_1},  \ldots, d^u_{i_k}, e_{i_{k+1}} , d^u_{i_{k+2}},
    , \ldots,  d^u_{i_m})\ .
\end{split}
\]
So in every single case we have 
\[
v_{i_j} \det (d^u_{i_1}, \ldots, d^u_{i_k}, e_{i_j},
d^u_{i_{k+2}}\!\!\!- d^u_{i_{k+1}}, \ldots,  d^u_{i_m}\!\!\!- d^u_{i_{k+1}} )\\
= -\det (d^u_{i_1}, \ldots, e_{i_j} , \ldots,  d^u_{i_m})\ .
\]
To complete the calculation we insert these results into \eqref{d}:
\[
\begin{split}
\mathcal{D} 
&= -t^{m-2} \sum_{j=1}^m \det (d^u_{i_1}, \ldots, e_{i_j} , \ldots,
d^u_{i_m})\\
&= \sgn(I^u_+, I^u_-) t^{m-2} \sum_{j=1}^m \det (d^u_1, \ldots, -e_j, \ldots,
d^u_m)\\
&= \sgn(I^u_+, I^u_-) t^{m-2} \chi_D^\prime (u).  
\end{split}
\]
We have $\chi^\prime(u) \neq 0$ since $u$ is simple. Therefore the
intersection is transversal and the index of intersection at the
point under consideration is $\sgn \mathcal{D}$.
\end{proof}

\begin{cor}
\label{cor:index}
Let $D$ be nonsingular, all its negative eigenvalues be
general and $\ell_-$ the number of negative eigenvalues. Denote $\tilde{h}(J):= h(\Delta_{J} \times [0,1])$. \\
Then we have
\begin{equation}
\label{eq:index_full}
\sum_{\{I_+, I_-\} \in \mathcal{P}} \!\!\sgn (I_+, I_-) \ \ \I
\big( \tilde{h}(I_+), \tilde{h}(I_-) \big)
\ \ = \ \ \left\{ 
  \begin{array}{l l} 
    0  & \mbox{ if } \det D > 0\ , \\
    -1 & \mbox{ if } \det D < 0\ .
  \end{array}
\right.  
\end{equation}
For every subset $S \subset \mathcal{P}$ we have 
\begin{equation}
\label{eq:index_part}
- \left \lceil \tfrac{\ell_-}{2} \right \rceil \ \ \leq \ \ 
\sum_{\{I_+, I_-\} \in S} \sgn (I_+, I_-) \ \I
\big( \tilde{h}(I_+), \tilde{h}(I_-) \big)
\ \ \leq\ \  \left \lfloor \tfrac{\ell_-}{2} \right \rfloor\ .
\end{equation}
As a special case we have for every individual pair $ \{I_+, I_-\} \in \mathcal{P}$ the estimates
\begin{equation}
- \left \lceil \tfrac{\ell_-}{2} \right \rceil \ \ \leq \ \ 
\sgn (I_+, I_-) \ \I
\big( \tilde{h}(I_+), \tilde{h}(I_-) \big)
\ \ \leq\ \  \left \lfloor \tfrac{\ell_-}{2} \right \rfloor\ .
\end{equation}
\end{cor}
\begin{figure} 
\input{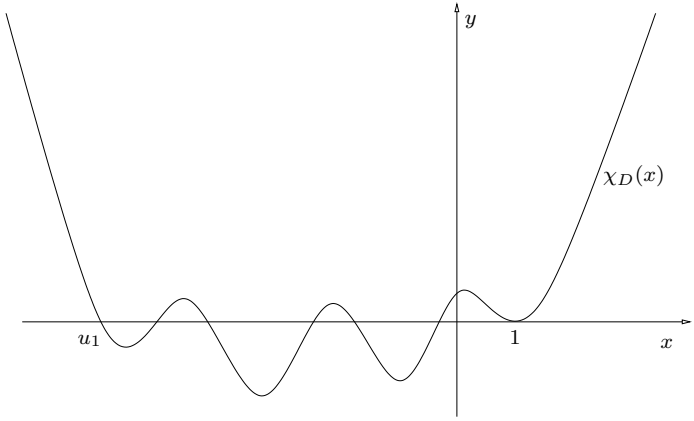}
\caption{A characteristic polynomial with simple negative roots}
\end{figure}
\begin{proof}
Intersection times $t \in [0,1]$ correspond to eigenvalues $u<0$ of $D$.
The first root $u_1$ of $\chi_D$ satisfies $\chi_D^\prime(u_1) < 0$
and two consecutive roots $u, \hat{u}$ of $\chi_D$ satisfy $\sgn
\chi^\prime_D(u) = -\sgn \chi^\prime_D(\hat{u})$. So $\chi_D$ has at most $\left \lceil
\frac{\ell_-}{2} \right \rceil$ negative roots $u$ with $\chi_D^\prime(u) < 0$
and  at most $\left \lfloor  \frac{\ell_-}{2} \right \rfloor$ negative
roots $\tilde{u}$ with $\chi_D^\prime(\tilde{u})>0 $. These are exactly the
terms in the sums above. \\
\end{proof}

\subsubsection{Application to the Deformation Cochain}

The relations between coefficients of $\lambda$  we develop  here are local in the sense that we only look at few vertices
at the same time. We restrict to subcomplexes of $\S$
consisting of $m+1$ points.
For a subset $J:= \{ j_0, \ldots j_m \}_< \subset [N]$ and $k \in \N$ denote 
$\ell_J^k \  := \ \# \big((J\setminus \{j_0\}) \cap [k]\big)$ and
\[
\mathcal{P}_J := \{\tau_+ \times \tau_- \in
\dprod{\S} \mid \dim (\tau_+ \times  \tau_-) \ = \ m-1, \tau_+
\cup \tau_- = J, j_0 \in \tau_+ \}.    
\]
These are the products of simplices with vertices in $J$ that we may insert
into the deformation cochain.

\begin{thm}[Related coefficients of the deformation cochain]
\label{thm:properties}
Let $f$ and $g$ be  general position maps of the vertex set
$\langle N \rangle$ of $\S$ into $\R^m$ and $k \in \langle N \rangle$. 
Assume further that $f(i) = g(i)$ for 
$i \in \langle k \rangle$ and that the set $\{f(0), \ldots, f(N), g(k+1), \ldots, g(N)\}$ is in
general position. \\ 
For every subset $J \subset \langle N\rangle$ with $|J| = m+1$ denote by $\varepsilon_g(J)$ the 
orientation of the simplex $g(J)$.
Then the deformation cochain $\lambda_{f,g} \in \mathcal{C}^{m-1}
(\dprod{\S})$ has the following
properties:\\ 
\begin{equation}
\label{eq:points}
\sum_{\tau_+ \times \tau_- \in \mathcal{P}_J} |\lambda_{f,g}(\tau_+
  \times \tau_-)| \ \ \leq \ \ m - \ell_J^k,
\end{equation}
\begin{equation}
\label{eq:index1}
-1 \ \ \leq \ \ 
\varepsilon_g(J) \sum_{\tau_+ \times \tau_- \in \mathcal{P}_J } \sgn (\tau_+, \tau_-) \lambda_{f,g} (\tau_+
  \times  \tau_-)
\ \ \leq\ \  0,
\end{equation}
and 
\begin{equation}
\label{eq:index2}
- \left \lceil \frac{m-\ell_J^k}{2} \right \rceil \ \ \leq \ \ 
\varepsilon_g(J) \sgn (\tau_+, \tau_-) \lambda_{f,g} (\tau_+
  \times  \tau_-)
\ \ \leq\ \  \left \lfloor \frac{m-\ell_J^k}{2} \right \rfloor
\end{equation}
for every $\tau_+ \times \tau_- \in  \mathcal{P}_J$. 
\end{thm}
\begin{proof}
Fix a subset $J:= \{ j_0, \ldots j_m \}_< \subset \langle N\rangle$.
Perform a basis transformation $A_J$ that takes $(g(j_0),0), \ldots
,(g(j_m),0), (f(j_0), 1)$ to $e_0, \ldots , e_{m+1}$ respectively. 
$\varepsilon_g(J)$ is the sign of the determinant of this basis transformation. 
%, because an orientation reversing basis transformation changes the 
% sign of the intersection numbers.
Let $(d_1, 1),
\ldots , (d_m,1)$ be the images of   $(f(j_1),1), \ldots
,(f(j_m),1)$ and $D := (d_1, \ldots , d_m)$. Denote $j: [m] \to J$, $i
\mapsto j_i$.
Then $h := A_J \circ
h_{f,g} \circ j$ is of the form we considered in Subsection
\ref{section:deformation}. Moreover the first $\ell_J^k$
columns of $D$ are $e_1, \ldots ,
e_{\ell_J^k}$. The eigenvalue $1$ has at least multiplicity
$\ell_J^k$. Thus $m -\ell_J^k$ is a upper bound for the number
of negative eigenvalues. Now
\begin{eqnarray*}
 \lambda_{f,g} (\tau_+ \times \tau_-) &=& \I \big( h_{f,g}(\tau_+ \times
[0,1]), h_{f,g}(\tau_- \times [0,1])\big)\\
&=& \varepsilon_g(J) \I \big( h(j^{-1}(\tau_+) \times
[0,1]), h(j^{-1}(\tau_-) \times [0,1])\big)\ 
\end{eqnarray*}
and therefore
\[
|\lambda_{f,g} (\tau_+ \times \tau_-)| \ \ \leq \ \ \# \left( h(j^{-1}(\tau_+) \times
[0,1] \cap h(j^{-1}(\tau_-) \times [0,1])\right)
\]
So we immediately get equation (\ref{eq:points}) from Corollary \ref{cor:points} and equations 
(\ref{eq:index1}) and (\ref{eq:index2}) from Corollary \ref{cor:index}.
\end{proof}
\begin{remark}  
\label{remark:novik}
Condition (\ref{eq:index2}) implies 
\[
- \left \lceil \frac{m}{2} \right \rceil \ \ \leq \ \ 
\lambda_{f,g} (\tau_+ \times  \tau_-)
\ \ \leq\ \  \left \lceil \frac{m}{2} \right \rceil
\]
which are the restrictions on the values 
of $\lambda_{f,g}$ that  Novik derived (cf. \cite[Theorem 3.1]{novik2000}).
\end{remark}

\section{Geometric Realizability and beyond}
\label{section:extensions}

Up to now we have looked at arbitrary general position
maps. In this section we compare a map with special properties such as
a geometric realization with a reference map whose intersection
cocycle can be easily computed.
 
\subsection{The Reference Map}

We start by defining our reference map:

Denote by $c: \langle N \rangle \to \R^m$ the \emph{cyclic map}
which maps vertex $i$ to the point $c(i) = (i, i^2, \ldots
i^m)^t$ on the moment curve.

\begin{prop}[{\cite[Lemma 4.2]{shapiro1957}}]
Let $k + \ell = m$, $k \geq \ell$, $s_0 < s_1 < \ldots < s_k$, $t_0 <
t_1 < \ldots < t_\ell$. If $k = \ell$ assume further that $s_0 < t_0$.
The two simplices $\sigma = \conv \{c(s_0), \ldots, c(s_k)\}$ and
$\tau = \conv \{c(t_0), \ldots, c(t_\ell)\}$ of complementary dimensions
intersect if and only if 
their dimensions differ at most by one and their vertices alternate along the curve:
\[
k = \left\lceil \frac{m}{2} \right\rceil \qquad \mbox{ and } \qquad s_0 < t_0 <
s_1 < \ldots < s_{\lfloor \frac{m}{2} \rfloor}< t_{\lfloor \frac{m}{2}
  \rfloor} (< s_{\lceil \frac{m}{2} \rceil})
\]
In the case of intersection we have
\[
\I(\sigma, \tau) = (-1)^{\frac{(k-1)k}{2}}.
\]
\end{prop}
 \begin{proof} For every set $\{c_0, \ldots, c_{m+1}\}$ consisting of
 $m+2$ points $c_i = c(u_i)$ with $u_0 < u_1 < \ldots < u_{m+1}$ there
 is a unique affine dependence 
\[
\sum_{i=0}^{m+1} \alpha_i c_i \ = \ \zerovec \qquad \mbox{with} \qquad
\sum_{i=0}^{m+1} \alpha_i \ = \ 0 \ \mbox{and} \ \alpha_0 = 1  .
\]
We calculate the sign of the coefficients $\alpha_k$.
\[
\begin{split}
\det \left( \tbinom{1}{c_0},\ldots, \widehat{\tbinom{1}{c_k}},
      \ldots, \tbinom{1}{c_{m+1}} \right)
&= - \alpha_k \det \left( \tbinom{1}{c_k}, \tbinom{1}{c_1}, \ldots,
  \widehat{\tbinom{1}{c_k}}, \ldots, \tbinom{1}{c_{m+1}}
\right)\\
&= (-1)^k \alpha_k \det \left(\tbinom{1}{c_1}, \ldots, \tbinom{1}{c_{m+1}}\right)
\end{split}
\]
Since $\det \left( \binom{1}{c_0}, \ldots, \widehat{\binom{1}{c_k}},
  \ldots, \binom{1}{c_{m+1}} \right)$ and $\det
\left(\binom{1}{c_1}, \ldots, \binom{1}{c_{m+1}} \right)$ are both 
positive we get 
\[
(-1)^k \alpha_k >0
\]
that is, 
\[
\sgn \ \alpha_k \ \ = \ \ \left\{ \begin{array}{l l}
                         +1 & \mbox{if } k \mbox{ is even}\\ 
                         -1 & \mbox{if } k \mbox{ is odd.} 
                        \end{array}
                \right.
\]
The proposition follows.
\end{proof}

\subsection{Deformation Cochains of Geometric Realizations}

If a simplicial maps defining the deformation cochain is a
simplicial \emph{embedding}, we know, that the images of certain simplices don't intersect.
The following trivial observation about the coefficients of deformation cochains is the key to 
bring in the combinatorics of the complex $\K$.
\begin{lemma}
\label{eq:forbidden}
Let $f, g: \langle N\rangle \to \R^m$ be general position maps.\\
If $f(\sigma) \cap f(\tau) = \emptyset$ and $\dim \sigma + \dim \tau =
m$ then
\begin{equation*}
\delta \lambda_{f,g} (\sigma \times \tau) \ \ = \ \ 
 - \varphi_g(\sigma \times \tau) \ .
\end{equation*}
\end{lemma}
\begin{proof}
$\quad\varphi_f (\sigma \times \tau) = (-1)^{\dim \sigma} 
                                 \I\big(f(\sigma), f(\tau)\big) = 0
$.
\end{proof}

\begin{remark}
The expression 
\[
\delta \lambda (\sigma \times \tau)\ \ = \ \ \sum_{i=0}^{\dim \sigma} (-1)^i \lambda (\sigma^i \times \tau) +
\sum_{j=0}^{\dim \tau} (-1)^{\dim \sigma +j} 
      \lambda (\sigma \times \tau^j) 
\]
is linear in the coefficients of the deformation cochain $\lambda$.
So for every pair $\sigma \times \tau$  of simplices of complementary
dimensions with disjoint images we get a linear equation that is valid 
for the coefficients of $\lambda_{f,g}$.
\end{remark}

This is particularly useful when we assume the existence of a
geometric realization but can also be used to express geometric immersability.
 So we gather all information we have on the 
deformation cochain in our main Theorem:

\begin{thm}[Obstruction Polytope]
\index{obstruction!polytope}
\index{obstruction!system}
\label{thm:system}
If there is a geometric realization of the simplicial complex ${\K}$ 
in $\R^m$
then the obstruction polytope in the cochain space $\mathcal{C}^{m-1}(\dprod{\S}, \R)$ given by the following 
inequalities contains a point $\lambda \in \mathcal{C}^{m-1}(\dprod{\S}, \Z)$ with integer coefficients.
\begin{enumerate}[\quad \rm 1.]
\item {\rm (The symmetries of Lemma \ref{lem:symmetries})}
\label{mt:sym}
\begin{equation*}
\lambda(\tau_1 \times \tau_2) =(-1)^{(\dim \tau_1 +1)(\dim \tau_2 +1)} \lambda_{f,g} (\tau_2 \times \tau_1)
\end{equation*}
for all $\tau_1 \times \tau_2 \in \dprod{\S}$,
\item {\rm (The deformation inequalities of Theorem \ref{thm:properties})}
\label{mt:deformation}
For every subset $J \subset \langle N\rangle$ 
\begin{enumerate}[\rm (a)]
\item \begin{equation*}
\sum_{\tau_+ \times \tau_- \in \mathcal{P}_J} |\lambda(\tau_+
  \times \tau_-)| \ \ \leq \ \ m - \ell_J^m,
\end{equation*}
\item 
\begin{equation*}
-1 \ \ \leq \ \ 
\sum_{\tau_+ \times \tau_- \in \mathcal{P}_J } \sgn (\tau_+, \tau_-) 
  \lambda(\tau_+\times  \tau_-)
\ \ \leq\ \  0 ,
\end{equation*}
and for every $\tau_+ \times \tau_- \in  \mathcal{P}_J$
\item 
\label{mt:bounds}
\begin{equation*}
- \left \lceil \frac{m-\ell_J^m}{2} \right \rceil \ \ \leq \ \ 
\sgn (\tau_+, \tau_-) \lambda(\tau_+
  \times  \tau_-)
\ \ \leq\ \  \left \lfloor \frac{m-\ell_J^m}{2} \right \rfloor,
\end{equation*} 
\end{enumerate}
\item {\rm (The intersection and linking inequalities of Proposition
  \ref{prop:intersection_inequalities})}
\begin{enumerate}[\rm (a)]
\item $\varphi_c(\sigma \times \tau) -1 \ \ \leq \ \  \delta
  \lambda(\sigma \times \tau) \ \ \leq\ \ \varphi_c(\sigma
  \times \tau) +1 $ \\
for all $\sigma \times \tau \in
\dprod{\S}$, $\dim \sigma +\dim \tau = m$, 
\item 
\label{mt:linking}
$\varphi_c(\sigma \times \del \tau) -1 \ \ \leq \ \ 
  \lambda(\del \sigma \times \del \tau) \ \ \leq\ \ \varphi_c(\sigma
  \times \del \tau) +1 $ \\
for all $\sigma \times \tau \in
 \dprod{\S}$, $\dim \sigma = m-1$ and $\dim \tau = 2$, 
\end{enumerate}
\item {\rm (The equations of Lemma \ref{eq:forbidden})} 
\label{mt:coboundary}
For every pair $\sigma \times \tau$ of simplices in
  $\dprod{\K}$:
\[
      \delta \lambda(\sigma \times \tau) \ \  = \ \  -\varphi_c(\sigma \times \tau)
\]
\end{enumerate} 
\end{thm}
\begin{proof}
If there is a geometric realization $f: \langle N \rangle \to \R^m$
then there also is a geometric realization $\tilde{f}$ such that the first $m$ vertices 
satisfy $\tilde{f}(i) = c(i)$ for $i \in \langle m \rangle$ and such that the set 
$\{\tilde{f}(0), \ldots ,\tilde{f}(N), c(m+1), \ldots , c(N)\}$ is in
general position. The deformation cochain $\lambda_{c, \tilde{f}}$ has
the desired properties as $\varepsilon_c \equiv 1$.
\end{proof}

\begin{remark}
The system I.~Novik described, consists of the equations \ref{mt:coboundary} along with the equations 
\ref{mt:sym} and the bounds from Remark \ref{remark:novik} for pairs of simplices in $\dprod{\K}$ only. 
\end{remark}

\section{Subsystems and Experiments}
\label{section:experiments}
 
Theorem~\ref{thm:system} provides us with a system of linear
equations and inequalities that has an integer solution if the
complex $\K$ has a geometric realization.
So we can attack non-realizability-proofs by solving integer programming 
feasibility problems.
However the system sizes grow rapidly with the number of vertices. There are $\mathcal{O}(n^{m+1})$ variables 
in the system associated to a complex with $n$ vertices and target ambient dimension $m$. For Brehm`s triangulated 
M\"obius strip (and all other complexes on 9 vertices) we already get 
1764 variables. The integer feasibility 
problems --- even for
complexes with few vertices --- are therefore much too big to be sucessfully solved with standard integer programming software. 
On the other hand for a non-realizability proof it suffices to exhibit a subsystem 
of the obstruction system that has no solution.

In this section  we therefore look at subsystems of the obstruction system, that only use those variables 
associated to simplices that belong to the complex $\K$ and certain sums of the other variables. 

\begin{sub}
\label{cor:variable_approach}
If there is a geometric realization of the simplicial complex $\K$ in $\R^m$
then there is a cochain $\lambda \in \mathcal{C}^{m-1}(\dprod{\K})$ that satisfies 
the equations of Lemma \ref{eq:forbidden} for every pair of simplices in $\dprod{\K}$
and the linking inequalities \eqref{mt:linking} of
Proposition \ref{prop:intersection_inequalities} that only use values of
$\lambda$ on $\dprod{\K}$.

The deformation inequalities of Theorem \ref{thm:properties} imply the following 
for the variables under consideration:
For every subset $J \subset \langle N\rangle$ we have
\begin{equation}
\sum_{\tau_+ \times \tau_- \in \mathcal{P}_J \cap \dprod{\K}}
  |\lambda(\tau_+\times \tau_-)|  + |y_J| \ \ \leq \ \ m - \ell_J^m,
\end{equation}
\begin{equation}
-1 \ \ \leq \ \ 
\sum_{\tau_+ \times \tau_- \in \mathcal{P}_J \cap \dprod{\K}} \sgn (\tau_+, \tau_-) \lambda(\tau_+ \times  \tau_-) +y_J 
\ \ \leq\ \  0
\end{equation}
by introducing the new variable $y_J$ for \lq the rest of the sum\rq.
We still have for every $\tau_+ \times \tau_- \in  \dprod{\K}$
\begin{equation}
- \left \lceil \frac{m-\ell_J^m}{2} \right \rceil \ \ \leq \ \ 
\sgn (\tau_+, \tau_-) \lambda_{f,g} (\tau_+
  \times  \tau_-)
\ \ \leq\ \  \left \lfloor \frac{m-\ell_J^m}{2} \right \rfloor.
\end{equation} 
and the same bounds hold for $y_J$.
\end{sub}

We can even do with less variables at the expense of more inequalities.

\begin{sub}
\label{sub:full} If there is a geometric realization of the simplicial complex $\K$ in $\R^m$
then there is a cochain $\lambda \in \mathcal{C}^{m-1}(\dprod{\K})$ that satisfies 
the equations of Lemma \ref{eq:forbidden} for every pair of simplices in $\dprod{\K}$
and the linking inequalities \eqref{mt:linking} of
Proposition \ref{prop:intersection_inequalities} that only use values of
$\lambda$ on $\dprod{\K}$.
The deformation inequalities of Theorem \ref{thm:properties} imply
inequalities for every subset $S$ of $\mathcal{P}_J \cap \dprod{\K}$.
\begin{equation}
- \left \lceil \frac{m-\ell_J^m}{2} \right \rceil \ \ \leq \ \ 
\sum_{\tau_+ \times  \tau_- \in S}\sgn (\tau_+, \tau_-) 
     \lambda_{f,g} (\tau_+ \times  \tau_-)
\ \ \leq\ \  \left \lfloor \frac{m-\ell_J^m}{2} \right \rfloor.
\end{equation}
and
\begin{equation}
\sum_{\tau_+ \times \tau_- \in \mathcal{P}_J \cap \dprod{\K}}
  |\lambda(\tau_+\times \tau_-)| \ \ \leq \ \ m - \ell_J^m,
\end{equation} 
\end{sub}

The systems of the above Corollaries are generated by the gap program 
\url{generate_obstructions.gap} that can be obtained via my homepage
\url{http://www.math.tu-berlin.de/~timmreck}.

The resulting systems can be examined further by integer programming software.
I ran several experiments using SCIP \cite{scip} to examine the resulting systems.
Table \ref{table:results} gives an overview on system sizes and solution times. 
$\mathcal{M}_g$ denotes an orientable surface of genus $g$. The triangulations under 
consideration have the minimum number of vertices and can be found in the file. 
$\mathcal{B}$ denotes the triangulated M\"obius strip by Brehm. 
The systems under consideration are those of Subsystem \ref{sub:full} expressing the inequalities 
involving absolute values without the use of new variables.

\begin{table}
{\small
\begin{tabular}{|c|c|c|c|c|c|c|c|}
surface & file & realizable & $f$-vector & var. & constr. & solv. & time\\
\hline
$\R P^2$ & rp2.gap & no & (6, 15, 10) & 150 & 1365 & no & 0.24 sec \\
$\mathcal{B}$ & moebius.gap & no \cite{brehm1983} & (9, 24, 15) & 510 & 2262 & no &46.3 sec \\
$\mathcal{M}_0$ & bipyramid.gap & yes & (5, 9, 6) & 48 & 500 & yes & 0.1 sec \\
$\mathcal{M}_1$ & csaszar.gap & yes & (7, 21, 14) & 322 & 2583 & yes & 0.78 sec \\
$\mathcal{M}_2$ & m2\_10.gap & yes \cite{lutz2004}& (10, 36, 24) & 1136 & 5888 & yes & 34.83 sec \\
$\mathcal{M}_3$ & m3\_10.gap & yes \cite{HougardyLutzZelke2006pre}& (10, 42, 28) & 1490 & 9847 & yes & 143 sec \\
$\mathcal{M}_4$ & m4\_11.gap & yes \cite{bokowski1989}& (11, 51, 34) & 2248 & 15234 & yes & 564 min\\
$\mathcal{M}_5$ & m5\_12.gap & ? & (12, 60, 40) & 3180 & 21840 & ? & \\
$\mathcal{M}_6$ & altshuler54.gap & no \cite{bokowski2000} & (12, 66, 44) & 3762 & 33473 & ? & \\
\end{tabular}
}
\caption{Computational results}
\label{table:results}
\end{table}

\bigskip
The smallest system showing the non-realizability of the M\"obius strip only uses the parts
\ref{mt:sym},  \ref{mt:bounds}, \ref{mt:linking} and \ref{mt:coboundary} of Theorem \ref{thm:system} 
and has 510 variables and 426 constraints.
The systems for genus 5 and 6 using only these parts of  Theorem \ref{thm:system} are solvable.

\providecommand{\bysame}{\leavevmode\hbox to3em{\hrulefill}\thinspace}
\providecommand{\MR}{\relax\ifhmode\unskip\space\fi MR }
\providecommand{\MRhref}[2]{%
  \href{http://www.ams.org/mathscinet-getitem?mr=#1}{#2}
}
\providecommand{\href}[2]{#2}

\label{Timmreck:end}
%\bibliographzstyle{amsplain}
%\bibliography{dagmar}
\end{document}